\documentclass[letterpaper,12pt]{amsart}

\title{Normal Binary Hierarchical Models}

\author{Daniel Irving Bernstein}
\author{Seth Sullivant}
\address{Department of Mathematics \\ North Carolina State University, Raleigh, NC 27695}
\email{dibernst@ncsu.edu}
\email{smsulli2@ncsu.edu}

\usepackage{latexsym,array,delarray,amsthm,amssymb,epsfig,graphicx,multirow,longtable}
\usepackage[numbers]{natbib}
\usepackage{tikz}
\usepackage{verbatim}
\usepackage{graphicx}

\theoremstyle{plain}
\newtheorem{thm}{Theorem}[section]
\newtheorem{lemma}[thm]{Lemma}
\newtheorem{prop}[thm]{Proposition}

\newtheorem*{thm*}{Theorem}
\newtheorem*{lemma*}{Lemma}
\newtheorem*{prop*}{Proposition}
\newtheorem*{cor*}{Corollary}
\newtheorem*{conj*}{Conjecture}

\theoremstyle{definition}
\newtheorem{defn}[thm]{Definition}
\newtheorem*{defn*}{Definition}
\newtheorem{ex}[thm]{Example}
\newtheorem{pr}[thm]{Problem}

\newtheorem{ques}[thm]{Question}

\theoremstyle{remark}

\newcommand{\zz}{\mathbb{Z}}
\newcommand{\nn}{\mathbb{N}}

\newcommand{\qq}{\mathbb{Q}}
\newcommand{\rr}{\mathbb{R}}

\newcommand{\kk}{\mathbb{K}}

\newcommand{\bfa}{\mathbf{a}}
\newcommand{\bfb}{\mathbf{b}}

\newcommand{\bfd}{\mathbf{d}}

\newcommand{\bfh}{\mathbf{h}}
\newcommand{\bfi}{\mathbf{i}}
\newcommand{\bfj}{\mathbf{j}}

\newcommand{\bfn}{\mathbf{n}}

\newcommand{\bfr}{\mathbf{r}}
\newcommand{\bft}{\mathbf{t}}

\newcommand{\bfv}{\mathbf{v}}

\newcommand{\bfx}{\mathbf{x}}
\newcommand{\bfy}{\mathbf{y}}
\newcommand{\bfz}{\mathbf{z}}

\newcommand{\cala}{\mathcal{A}}

\newcommand{\calc}{\mathcal{C}}
\newcommand{\cald}{\mathcal{D}}

\newcommand{\ind}{\mbox{$\perp \kern-5.5pt \perp$}}

\newcommand{\link}{\textnormal{link}}

\newcommand{\cone}{\textnormal{cone}}
\newcommand{\facet}{\textnormal{facet}}

\tikzstyle{vertex}=[circle, draw, inner sep=0pt, minimum size=6pt, fill=black]

\begin{document}

\begin{abstract}
Each simplicial complex and integer vector yields a vector
configuration whose combinatorial properties are important for
the analysis of contingency tables.  We study the normality of these
vector configurations including a description of operations on
simplicial complexes that preserve normality, constructions
of families of minimally nonnormal complexes, and computations
classifying all  of the normal complexes on up to six vertices.
We repeat this analysis for compressed vector configurations,
classifying all of the compressed complexes on up to six vertices.
\end{abstract}

\maketitle

\section{Introduction}
Associated to a simplicial complex $\calc$ with ground set $[m] := 
\{1, 2, \ldots, m\}$
and an integer vector $\bfd \in \zz^m$ is an integral matrix $\mathcal{A}_{\calc, \bfd}$ (defined in Section \ref{sec:preliminaries}).
The hierarchical model associated to $\calc, \bfd$ is a log-linear model (i.e.~discrete exponential family,
i.e.~toric variety) whose design matrix is the matrix $\mathcal{A}_{\calc, \bfd}$.
The linear transformation represented by the matrix $\mathcal{A}_{\calc, \bfd}$
takes an array $u \in \rr^{d_{1} \times \cdots \times d_{m}}$
and computes lower order marginals of $u$ according to the faces
of the simplicial complex $\calc$.
As for all log-linear models,
important relevant problems are to study properties of the lattice $\ker_\zz \mathcal{A}_{\calc, \bfd}$,
the polyhedral cone $\rr_{\geq 0} \mathcal{A}_{\calc, \bfd} := \{\mathcal{A}_{\calc, \bfd}x : x \geq 0 \}$ 
and the semigroup $\nn \mathcal{A}_{\calc, \bfd} :=  \{\mathcal{A}_{\calc, \bfd}x : x \geq 0, x \mbox{ integral} \}$.
In the special case where $\bfd = {\bf 2}$ the vector of all twos, we call $\mathcal{A}_{\calc, {\bf 2}}$ a binary hierarchical model.

A matrix $A \in \rr^{d \times n}$ is called \emph{normal} if
$\rr_{\geq 0}A  \cap \zz A  =  \nn A$.
For general matrices $A$, deciding membership in the set $\nn A$ is an
NP-complete problem.  However, for normal $A$ the problem is polynomial
time solvable since membership in both $\rr_{\geq 0} A$ and $\zz A$
have polynomial time solutions (via linear programming and Smith normal form,
respectively).  So developing methods to decide a priori whether or not a
matrix is normal can be useful for various problems related to
integer programming.  In this paper we focus on the following fundamental normality
question for the matrices of hierarchical models.

\begin{ques}\label{ques:main}
	For what values of $\calc,\bfd$ is the matrix $\mathcal{A}_{\calc,\bfd}$ normal?
\end{ques}

If $\mathcal{A}_{\calc,\bfd}$ is normal and $\bfd' \le \bfd$, then $\mathcal{A}_{\calc,\bfd'}$ is also normal,
so it makes sense to focus on the binary case of Question \ref{ques:main} as a first step.
Part of our motivation for studying this problem comes from
previous joint work of  Rauh and the second author \cite{rauh2014}, where normality was identified as a key property of hierarchical models to be able to apply the toric fiber product  construction to calculate a Markov basis.

The special case where the underlying simplicial complex
$\calc$ is a graph and $\bfd = {\bf 2}$ was
solved in \cite{sullivant2010} where it was shown that a binary graph
model is normal if and only if the graph is free of $K_4$ minors.
A complete classification is also available for the special
case where $\calc$ is the boundary of a simplex and $\bfd$ is arbitrary \cite{Bruns2011}.

In this paper we prove some new results and collect known results that relate to the classification of normal hierarchical models.
We also perform
extensive computations to check normality of hierarchical models where
$\calc$ has a small number of vertices.
Section \ref{sec:preliminaries} gives a detailed introduction to hierarchical models, the construction
and interpretation of the matrix $\mathcal{A}_{\calc, \bfd}$, and reviews the definition of normality.
Section \ref{sec:bigfacet} shows how the classification of unimodular complexes \cite{bernstein-sullivant2015}
extends to a classification of all normal binary models on $m$ vertices whose simplicial complex  has a facet with $m-1$ vertices.
Section \ref{sec:operations} reviews several operations on the simplicial complex that are known to preserve normality and also proves the new result
that normality is preserved on taking a vertex link.
Section \ref{sec:compressed} gives the same results, but for the property of being compressed.
Section \ref{sec:minimal} explores constructions for minimally nonnormal
complexes.
Section \ref{sec:computations} surveys our computational experiments to classify
complexes that give normal and compressed hierarchical models on small numbers of vertices.


\section{Preliminaries on Hierarchical Models and Normality}\label{sec:preliminaries}
	In this section, we explain how to construct the matrix $\mathcal{A}_{\calc,\bfd}$ associated to a hierarchical model and cover background material
	on normality of affine semigroups.
	
	\begin{defn}
		Let ${\bf d} = (d_1,\dots,d_n) \in \zz^n$ be such that $d_i \ge 2$ for each $i$.
		Then we define $\rr^{\bf d}$ to be the vector space of all $d_1\times \dots \times d_n$-way tables.
		That is, each ${\bf u} \in \rr^{\bf d}$ is a table in $n$ dimensions where the $i$th dimension has $d_i$ levels.
		We denote entries in a table ${\bf u}$ by $u_{i_1,\dots,i_n}$ where $(i_1,\dots,i_n) \in [d_1]\times\dots\times [d_n]$.
		We will use the shorthand ${\bf i}$ to denote $(i_1,\dots,i_n).$
	\end{defn}
	
	\begin{defn}
		Let $\mathcal{C}$ be a simplicial complex on ground set $[n]$.
		A \emph{facet} of $\mathcal{C}$ is an inclusion-maximal subset $F \subseteq [n]$
		that is contained in $\mathcal{C}$.
		We let $\facet(C)$ denote the collection of facets of $\mathcal{C}$.
	\end{defn}
	
	\begin{defn}
		Let ${\bf d} = (d_1,\dots,d_n) \in \zz^n$ and $F = \{f_1,\dots,f_k\} \subseteq [n]$.
		We define ${\bf d}_F$ to be the restriction of ${\bf d}$ to the indices in $F$,
		i.e. ${\bf d}_F = (d_{f_1},\dots,d_{f_k})$.
	\end{defn}
	
	\begin{defn}
		Let $\mathcal{C}$ be a simplicial complex on ground set $[n]$
		and let ${\bf d} = (d_1,\dots,d_n) \in \zz^n$ be such that $d_i \ge 2$ for each $i$.
		Then we have the linear map
		\[
			\pi_{\calc,{\bf d}}: \rr^{\bf d} \rightarrow \bigoplus_{F \in \facet(C)} \rr^{{\bf d}_F}
		\]
		defined by
		\[
			(u_{i_1, \ldots, i_n}  :  i_j \in [d_j])  \mapsto  \bigoplus_{\substack{{F \in {\rm facet}(\calc)}\\F=\{f_1,\dots,f_k\}}}
			\left( \sum_{\substack{{\bf i} \in [d_1]\times\dots\times[d_n] 
					\\:{\bf i}_F = {\bf j}}}  u_{i_1, \ldots, i_n} : {\bf j} \in [d_{f_1}]\times\dots\times[d_{f_k}]  \right).
		\]
	
		The matrix $\mathcal{A}_{\calc,{\bf d}}$ denotes the matrix in the standard basis
		that represents the linear transformation $\pi_{\mathcal{C},{\bf d}}$.
	\end{defn}
	
	We give an example of this construction.
	
	\begin{ex}		
	Let $\calc$ be the simplicial complex on ground set $[3] = \{1,2,3\}$ with
	${\rm facet}(\calc) =  \{ \{1\}, \{2,3\} \}$.
	The linear transformation $\pi_{\calc, \bfd}$
	maps a three-way tensor $u = (u_{ijk} : i \in [d_1], j \in [d_2], k \in [d_3] ) \in \rr^D$
	to the direct sum of a one-way tensor and a two-way tensor:
	$$
	\pi_{\calc, \bfd}( u)  =    ( \sum_{j,k}  u_{ijk} :  i \in [d_1])  \oplus 
	( \sum_{i}  u_{ijk} :  j \in [d_2], k \in [d_3]). 
	$$
	Taking $\bfd = \bf2 = (2,2,2)$, the matrix 
	$\mathcal{A}_{\calc, {\bf d}}$ could be represented
	as follows:
	$$
	\begin{pmatrix}
	1 & 1 & 1 & 1 & 0 & 0 & 0 & 0 \\
	0 & 0 & 0 & 0 & 1 & 1 & 1 & 1 \\
	\hline
	1 & 0 & 0 & 0 & 1 & 0 & 0 & 0 \\
	0 & 1 & 0 & 0 & 0 & 1 & 0 & 0  \\
	0 & 0 & 1 & 0 & 0 & 0 & 1 & 0 \\
	0 & 0 & 0 & 1 & 0 & 0 & 0 & 1  \\
	\end{pmatrix},
	$$
	the horizontal line dividing the image space into the parts corresponding
	to the two facets $\{1\}$ and $\{2,3\}$ of $\calc$.
	\end{ex}
	
	It can be useful to be precise about how to write down the matrix
	$\mathcal{A}_{\calc, \bfd}$.  The columns of the matrix
	$\mathcal{A}_{\calc, \bfd}$ are indexed by elements of the set
	$\prod_{j \in [n]}  [d_j]$ and the rows are indexed by pairs
	$(F, e)$ where $F \in {\rm facet}(\calc)$ and $e \in \prod_{j \in F}  [d_j]$.
	The entry in position with row index $(F, e)$ and column indexed by $i \in \prod_{j \in [n]}  [d_j]$ is $1$ if $e  =  (i_j :  j \in F)$, otherwise it is equal to
	$0$.
	We are interested in whether or not this matrix is normal.
	
	\begin{defn}\label{def:normal}
	Let $A \in \qq^{d\times n}$.
	We use the following notation for the lattice, cone, and semigroup spanned by the columns of $A$:
	\[
		\zz A := \{A\bfz : \bfz \in \zz^n\} \qquad \rr_{\ge 0} A := \{A\bfr : \bfr \in \rr_{\ge0}^n\} \qquad \nn A := \{A\bfn : \bfn \in \nn^n\}.
	\]
	Note that $\nn A \subseteq \zz A \cap \rr_{\ge 0} A$.
	We say the $\bfh$ is a \emph{hole} of $A$ if $\bfh \in \zz A \cap \rr_{\ge 0} A$ but $\bfh \notin \nn A$.
	If $A$ is free of holes, we say that $A$ is \emph{normal}.
\end{defn}

Not much is known about normality of the matrices $A_{\calc, \bfd}$.
To end this section, we survey some known results from the literature 
on normality of hierarchical models.

\begin{thm}\cite[{ Theorem 1}]{sullivant2010}
Let $\calc$ be a graph and $\bfd = {\bf 2}$.  Then $A_{\calc, \bfd}$
is normal if and only if $\calc$ is $K_{4}$ minor free.
\end{thm}

\begin{thm}\cite[{ Theorem 2.2}]{Bruns2011}\label{triangle}
Let $\calc$ be the simplicial complex whose facets
are all $m-1$ element subsets of $[m]$.  Then $A_{\calc,\bfd}$ is normal
in precisely the following situations up to symmetry:
\begin{enumerate}
\item At most two of the $d_{v}$ are greater than $2$.
\item $m = 3$ and $\bfd = (3,3,d_{3})$ for any $d_{3} \in \nn$.
\item $m = 3$ and $\bfd = (3,4,4), (3,4,5)$, or $(3,5,5)$.
\end{enumerate}
\end{thm}

\begin{thm}\cite[{Theorem 47 and following remarks}]{rauh2014} \label{thm:4cycle}
Let $\calc = [12][23][34][14]$ be the four-cycle graph.  Then
$A_{\calc, \bfd}$ is normal if $\bfd = (2,d_{2},2,d_{4})$
or $\bfd = (2,d_{2}, 3, d_{4})$ with $d_{2}, d_{3} \in \nn$.
\end{thm}


\section{Complexes with a Big Facet}\label{sec:bigfacet}

	The classification of unimodular binary hierarchical models given in \cite{bernstein-sullivant2015}
	easily extends to a classification of normality of all binary hierarchical models
	 containing a big facet (defined below).
	Throughout this section we assume that $\bfd = (2,\dots,2)$
	and we abbreviate $\mathcal{A}_{\calc,\bfd}$ by $\mathcal{A}_\calc$.
	We also abuse language and say ``$\calc$ is normal/unimodular''
	to mean ``$\mathcal{A}_\calc$ is normal/unimodular.''
	
\begin{defn}\label{def:unimodular}
		A matrix $A \in \zz^{d \times n}$ is \emph{unimodular} if 
the polyhedron $P_{A,b} = \{x \in \rr^s : Ax = b, x \ge 0\}$
		 has all integral vertices for every $b \in \zz A \cap \rr_{\geq 0}A$.
\end{defn}

\begin{prop}\label{uninormal}
		If $A \in \zz^{d\times n}$ is unimodular, then $A$ is normal.
\end{prop}

	\begin{defn}
		For any matrix $A \in \rr^{s \times t}$, we define the \emph{Lawrence lifting} of $A$ to be the matrix
		\[
			\Lambda(A) \quad =  \quad
				\begin{pmatrix}
					A & {\bf 0}\\
					{\bf 0} & A\\
					{\bf 1} & {\bf 1}
				\end{pmatrix}  \quad \in \quad   \rr^{ (2s + t)  \times 2t}
		\]
		where ${\bf 0}$ denotes the $s \times t$ matrix of all zeroes and 
		${\bf 1}$ denotes a $t \times t$ identity matrix.
		The kernel of this matrix is
		\[
			\{(u,-u): u \in \ker(A)\}.
		\]
		When a matrix $M$ has a kernel of this form, we say that $M$ is \emph{of Lawrence type}.
	\end{defn}

	\begin{defn}\label{def:lawrence}
		Let $\mathcal{C}$ be a simplicial complex on $[n]$.
		We define the \emph{Lawrence lifting} of $\mathcal{C}$
		to be the simplicial complex $\Lambda(\mathcal{C})$ on $[n+1]$ that has the following set of facets:
		\[
			\{ [n] \} \cup \{F \cup \{n+1\} : F \textnormal{ is a facet of } \mathcal{C}\}.
		\]
		In this case, we refer to the facet $[n]$ as a \emph{big facet}.
		Any simplicial complex $\mathcal{D}$ that contains a big facet $F$ (i.e. a facet containing all but one vertex)
		is the Lawrence lifting of the complex $\link_v(\mathcal{D})$ where $v$ is the vertex of $\mathcal{D}$ not contained in $F$  (see Definition \ref{defn:link} for
		the definition of $\link_v(\mathcal{D})$).
		In this case we say that $\mathcal{D}$ is \emph{of Lawrence type}.
	\end{defn}
	
	The following proposition justifies the multiple definitions of ``Lawrence lifting.''

	\begin{prop}\label{lawrenceform} \cite[Theorem 15]{Santos2003}
		$\Lambda(\mathcal{A}_{\calc}) = \mathcal{A}_{\Lambda(\calc)}$.
	\end{prop}

	\begin{prop}\label{lawrencenormalmatrix}  \cite[Corollary 5]{ohsugi-hibi2010}
		Let $A$ be an integer matrix of Lawrence type.
		Then $A$ is normal if and only if $A$ is unimodular.
	\end{prop}

	Combining Propositions \ref{lawrenceform} and \ref{lawrencenormalmatrix} gives the following theorem.
	
	\begin{thm}\label{bigfacettheorem}
		Let $\Lambda(\calc)$ be a simplicial complex of Lawrence type.
		Then $\Lambda(\calc)$ is normal if and only if $\calc$ is unimodular.
	\end{thm}
	\begin{proof}
		By Propositions \ref{lawrenceform} and \ref{lawrencenormalmatrix}, we know that $\Lambda(\calc)$ must be unimodular.
		Then \cite{bernstein-sullivant2015} Proposition 3.14 gives that $\Lambda(\calc)$ is unimodular if and only if $\calc$ is unimodular.
	\end{proof}
	
	See \cite{bernstein-sullivant2015} for a complete classification of unimodular $\calc$.
	Combining this with Theorem \ref{bigfacettheorem} gives a complete classification of the normal simplicial complexes
	of Lawrence type.


\section{Normality Preserving Operations}\label{sec:operations}

	An essential ingredient in the classification results from \cite{sullivant2010} and \cite{bernstein-sullivant2015}
	was a list of simplicial complex operations that preserve normality and unimodularity, respectively.
	This section describes the operations that are known to preserve normality.
	In particular, we list some results from \cite{sullivant2010}; namely that normality is preserved 
	by taking induced sub-complexes, contracting edges, and gluing two complexes along a face.
	We then provide a new normality-preserving operation - taking links of vertices.
	For completeness sake, we also describe adding cone vertices and ghost vertices which are easily seen to preserve normality.

	\begin{defn}\label{edgecontraction}
		Let $\mathcal{C}$ be a simplicial complex on vertex set $V$ and let $L \in \mathcal{C}$.
		We define the \emph{edge contraction} to be the simplicial complex	on $\{v\} \cup V \setminus L$ by
		\[
			\calc / L := \{S \in \calc : S \cap L = \emptyset\} \cup \{(S \setminus L) \cup \{v\} : S \cap L \neq \emptyset\}.
		\]
	\end{defn}
	
	\begin{defn}
		Let $\mathcal{C}$ be a simplicial complex on vertex set $V$ and let $v \in V$.
		We define the \emph{vertex deletion} to be the simplicial complex on $V\setminus \{v\}$ by
		\[
			\calc\setminus v = \{S \in \calc : v\notin S\}.
		\]
	\end{defn}
	
	\begin{lemma}[\cite{sullivant2010}, Lemma 2.1]\label{deletioncontraction}
		Suppose that $\mathcal{C}'$ is obtained from $\mathcal{C}$ by either
		\begin{enumerate}
			\item deleting a vertex ($\calc' = \calc \setminus v$) or
			\item contracting an edge ($\calc' = \calc / L$).
		\end{enumerate}
		Then if $\mathcal{A}_{\mathcal{C},\bfd}$ is normal, so is $\mathcal{A}_{\mathcal{C}',\bfd'}$,
		where $\bfd'$ is obtained from $\bfd$ by either (1) 
		deleting $d_{v}$ where $v$ is the deleted vertex or (2) 		deleting
		all $d_{i}$ such that $i \in L$ and setting
		 $d_v = \min_{i \in L} d_i$ in the case of a edge contraction.
	\end{lemma}
	
	\begin{defn}
		Let $\calc$ be a simplicial complex on vertex set $V$.
		If $\calc_1,\calc_2$ are simplicial complexes on subsets of $V$ and $S \subseteq V$,
		then we say that $\mathcal{C}$ has \emph{reducible decomposition} $(\calc_1,S,\calc_2)$ if
		$\calc = \calc_1\cup \calc_2$ and $\calc_1\cap\calc_2 = 2^S$.
		If $\calc$ has a reducible decomposition, we say that $\calc$ is \emph{reducible}.
	\end{defn}
	
	\begin{lemma}[\cite{sullivant2010}, Lemma 2.4]\label{reducible}
		Let $\calc$ be a reducible simplicial complex with decomposition $(\calc_1,S,\calc_2)$.
		Assume that $\bfd_1$ and $\bfd_2$ agree at all vertices in $S$.
		Let $\bfd$ be the vector obtained by identifying the corresponding entries of $\bfd_1,\bfd_2$.
		If $\mathcal{A}_{\calc_1,\bfd_1},\mathcal{A}_{\calc_2,\bfd_2}$ are normal,
		then $\mathcal{A}_{\calc,\bfd}$ is normal.
	\end{lemma}
	
	\begin{defn}\label{defn:link}
		Let $S  \in \mathcal{C}$ be a face of $\calc$.
		Then the \emph{link} of $S$ in $\mathcal{C}$ is the new simplicial complex
		\[
			\link_S( \mathcal{C} ) = \left\{F\setminus S : F \in \mathcal{C} \mbox{ and } S \subseteq F \right\}.
		\]
		When $S = \{v\}$, we simply write $\link_v(\mathcal{C}) := \link_{\{v\}}( \calc)$.
	\end{defn}
	
	Note that we can obtain $\link_S(\mathcal{C})$ by repeatedly taking links with respect to vertices.
	That is if $S$ is a face of $\calc$ and  $\#S \ge 2$ and $v \in S$, then
	\[
		\link_S(\mathcal{C}) = \link_{v}(\link_{S\setminus\{v\}}(\mathcal{C})).
	\]
	
	We show that taking links of vertices preserves normality after
	proving establishing a definition and proving two propositions.
	
	\begin{defn}
		For $\bfv \in \kk^r$, we let $P_\bfv$ denote the matrix that projects
		onto the hyperplane orthogonal to $\bfv$:
		\[
			P_\bfv = I - \frac{\bfv \cdot \bfv^T}{\|\bfv\|^2}.
		\]
	\end{defn}
	
	\begin{prop}\label{matrixlemma}
		Let $\bfa \in \rr^r$ be a nonzero vector.
		Let $T$ be a square matrix with $k^{\rm th}$ row $\frac{\bfa^\intercal}{\|\bfa\|^2}$ such that the remaining rows
		span the hyperplane $\{\bfx : \bfa^T\bfx = 0\}$.
		Then the following are true
		\begin{enumerate}
			\item $T$ is invertible
			\item $T\bfa = e_k$
			\item {$P_\bfa = T^{-1}P_{e_{k}}T$.}
		\end{enumerate}
	\end{prop}
	\begin{proof}
		The first two properties are clear so we proceed to prove the third.
		Without loss of generality, assume $k = 1$ and so {$\frac{\bfa^\intercal}{\|\bfa\|^2}$} is the top row of $T$.
		Since the first row $\frac{\bfa^\intercal}{\|\bfa\|^2}$ of $T$ is orthogonal to the rest,
		the first column of $T^{-1}$ must be $\bfa$.
		Note that the matrix $P_{e_1}T$ is $T$ with the top row replaced by $0$s.
		Therefore
		\[
			T^{-1}P_{e_1}T = I - \frac{\bfa\cdot\bfa^\intercal}{\|\bfa\|^2}
		\]
		which is the definition of $P_\bfa$.
	\end{proof}
	
	\begin{prop}\label{projection}
		Let $A \in \rr^{r\times n}$ be a normal matrix with columns $\{{\bf a}_i\}_{i=1}^n$.
		Then {$P_{a_n}A$} is normal.
	\end{prop}
	\begin{proof}
		First we show that the proposition is true when $A$ is of the following form
		\begin{equation}\label{eq:easy}
			\left(\begin{array}{ccc|c}
				t_1 & \cdots & t_{n-1} & t_n \\ \hline
				\multicolumn{3}{c}{\multirow{3}{*}{\raisebox{-7mm}{\scalebox{1}{{$B$}}}}} \vline&0 \\ 
				& & & \raisebox{2pt}{\vdots}\\
				& & & 0
	 		\end{array}\right)
		\end{equation}
		where $t_n \ge 0$.
		{If $t_n = 0$, then $P_{a_n}$ is the identity so $P_{a_n}A = A$ which is normal.
		So assume $t_n > 0$.
		Let $\binom{0}{\bfb} \in \zz P_{a_n}A \cap \rr_{\ge0}P_{a_n}A$ (note that $P_{a_n}A$ has zeros along the top row).}
		So there exist $\bfz \in \zz^{n-1}$ and $\bfr \in \rr^{n-1}_{\ge 0}$ such that $B\bfz = B\bfr = \bfb$.
		We lift $\bfb$ to $\hat{\bfb} \in \zz A \cap \rr_{\ge 0} A$ as follows.
		Define $\bft = (t_1,\dots,t_{n-1})$.
		Choose $d \in \zz$ such that $\langle \bft, \bfz \rangle - \langle \bft, \bfr \rangle + dt_n > 0$.
		Define
		\[
			\hat\bfb := A \begin{pmatrix}\bfz \\ d\end{pmatrix}
		\]
		so $\hat \bfb \in \zz A$.
		Define
		\[
			e := \frac{\langle \bft, \bfz \rangle - \langle \bft, \bfr \rangle + dt_n}{t_n}
		\]
		and so $\hat \bfb$ is in $\rr_{\ge 0} A$ because
		\[
			\hat\bfb := A \begin{pmatrix}\bfr \\ e\end{pmatrix}.
		\]
		Since $A$ is normal, there exists $\bfn \in \nn^n$ such that $A\bfn = \hat \bfb$.
		{ Then
		$P_{a_n}A\bfn = \binom{0}{\bfb}$ and so $\binom{0}{\bfb}$ is not a hole of ${P_{a_n}}A$.}
		So ${P_{a_n}}A$ is normal.
		\\
		\indent
		Now we let $\bfa_n$ be an arbitrary nonzero vector.
		If $T$ is an invertible matrix, then $TA$ is normal.
		To see this, note that the solution sets to the equations $TA\bfx = \bfb$ and $A\bfx = T^{-1}\bfb$ are identical.
		So $\bfb$ is a hole of $TA$ if and only if $T^{-1} \bfb$ is a hole of $A$.
		Let $T$ be a matrix whose first row is $\frac{\bfa_n^\intercal}{\|\bfa\|^2}$ and whose remaining rows are orthogonal to $\bfa_n^\intercal$.
		Since $T$ is invertible, $TA$ is normal.
		Furthermore, Proposition \ref{matrixlemma} (2) implies that the last column of $TA$ is $e_1^\intercal$.
		So by previous arguments, ${P_{a_n}}TA$ is also normal.
		Therefore, so is $T^{-1}{P_{a_n}}TA$.
		This is equal to $P_{\bfa_n}A$ by Proposition \ref{matrixlemma} (3).
	\end{proof}
	
	\begin{thm}\label{link}
		Assume $\mathcal{A}_{\mathcal{C},{\bf d}}$ is normal.
		Let $v$ be a vertex of $\mathcal{C}$ and let ${\bf d}'$ denote the vector obtained by deleting $d_{v}$ from ${\bf d}$.
		Then $\mathcal{A}_{\link_v{\mathcal{C}},{\bf d}'}$ is normal.
	\end{thm}
	\begin{proof}
		Let $A = \mathcal{A}_{\link_v{\mathcal{C}},{\bf d}'}$ and let $B = \mathcal{A}_{\mathcal{C}\setminus v, {\bf d}'}$.
		By Lemma 2.2 in \cite{hosten-sullivant2007} and the following remark, we can write
		\[
			\mathcal{A}_{\mathcal{C},{\bf d}} =\left[\begin{matrix}
							A &  {\bf 0}  & \ldots & {\bf 0}\\
							{\bf 0}  &  A & \ldots & {\bf 0}\\
							\vdots & \vdots & \ddots & \vdots\\
							{\bf 0}  &   {\bf 0}       &\ldots & A\\
							B & B & \dots & B
						\end{matrix}\right].
		\]
		Assume $A \in \rr^{m \times n}$ and $B \in \rr^{l \times n}$ and so $\mathcal{A}_{\mathcal{C},{\bf d}} \in \rr^{dm+l, dn}$.
		Without loss of generality, we may assume that $A,B,\mathcal{A}_{\mathcal{C},{\bf d}}$ all have full row rank.
		Let $\mathcal{A}'$ denote the matrix that results when we project $\mathcal{A}_{\mathcal{C},{\bf d}}$
		onto the subspace orthogonal to the last $(d-1)n$ columns.
		{By Proposition \ref{projection}, $\mathcal{A}'$ is normal.}
		Then if $1 \le i \le n$, the $i$th column of $\mathcal{A}'$ can be expressed as the $i$th column of $\mathcal{A}_{\mathcal{C},{\bf d}}$
		minus a linear combination of the last $(d-1)n$ columns of $\mathcal{A}_{\mathcal{C},{\bf d}}$,
		all of which are $0$ in the top $m$ rows.
		So this means that the top $m$ rows of $\mathcal{A}'$ are
		\[
			\begin{pmatrix} A & {\bf 0} & \dots & {\bf 0} \end{pmatrix}.
		\]
		Furthermore, since $A$ and $\mathcal{A}_{\mathcal{C},{\bf d}}$ both have full row rank,
		the final $(d-1)n$ columns of $\mathcal{A}_{\mathcal{C},{\bf d}}$ have rank $(d-1)m+l$.
		Therefore $\mathcal{A}'$ has rank $m$.
		Since $A$ also has rank $m$, we may delete the bottom $(d-1)m+l$ rows of $\mathcal{A}'$ without affecting the rowspace
		and therefore normality.
		So the matrix $\begin{pmatrix} A & {\bf 0} & \dots & {\bf 0} \end{pmatrix}$, and therefore $A$, is normal.
	\end{proof}
	
	For completeness sake, we give two very simple normality-preserving constructions.
	
	\begin{defn}
        	Let $\mathcal{C}$ be a simplicial complex on vertex set $[n]$. Then we define $\cone(\mathcal{C})$
        	to be the simplicial complex on $[n+1]$ with the following facets
        	\[
        		\{F \cup \{n+1\} : F \text{ is a facet of } \mathcal{C} \}
        	\]
		and we refer to this complex as a \emph{cone over $\mathcal{C}$} with $n+1$ as the \emph{cone vertex}.
	\end{defn}

        \begin{defn}
        	Let $\mathcal{C}$ be a simplicial complex on ground set $[n]$.
        	Let $G(\mathcal{C})$ denote the same simplicial complex but on ground set $[n+1]$.
        	Note that the vertex $n+1$ is not contained in any face of $G(\mathcal{C})$.
        	In this case we say that { $n+1$} is a \emph{ghost vertex}.
	\end{defn}
	
	\begin{prop}
		Let $\mathcal{C}$ be a simplicial complex on $[n]$ and let $\mathcal{C}'$ be $\cone(\mathcal{C})$ or $G(\mathcal{C})$.
		Let $\bfd$ be an integer vector of length $n$, let $\bfd'$ be { an extension of $\bfd$} of length $n+1$.
		Then $\mathcal{A}_{\calc,\bfd}$ is normal if and only if $\mathcal{A}_{\calc',\bfd'}$ is normal.
	\end{prop}
	\begin{proof}
		If $\calc' = \cone(\calc)$ then
		\[
			\mathcal{A}_{\calc',\bfd'} = \left(\begin{array}{c|c|c|c}
			\mathcal{A}_{\calc,\bfd} & {\bf 0} & \dots & {\bf 0} \\ \hline
			{\bf 0} & \mathcal{A}_{\calc,\bfd}& \dots & {\bf 0} \\ \hline
			\vdots & \vdots & \vdots & \vdots \\ \hline
			{\bf 0} & \dots & {\bf 0} & \mathcal{A}_{\calc,\bfd}
			\end{array}\right).
		\]
		If $\calc' = G(\calc)$ then
		\[
			\mathcal{A}_{\calc',\bfd'} = \begin{pmatrix}\mathcal{A}_{\mathcal{C},\bfd} & \dots & \mathcal{A}_{\mathcal{C},\bfd}\end{pmatrix}. \]
Both of these matrices are normal if $A_{\calc, \bfd}$ is.
	\end{proof}

As a final tool for proving normality of a complex,
we describe a condition on the polyhedral cone {$\rr_{\geq 0}\mathcal{A}_{\calc}$}
that allows us to remove a facet of a normal complex $\calc$ and retain normality.
To describe this construction, which is a special
case of Lemma { 2} in \cite{sullivant2010}, we need an alternate
representation of the matrix $\cala_{\calc}$.

\begin{defn}[Full-rank representation of $\cala_{\calc}$]\label{def:fullrank}
	Let $\calc$ be a simplicial complex on the set $\{1,2,\dots,n\}$.
	For each binary $n$-tuple $\bfb = b_1b_2\dots b_n \in \{0,1\}^n$
	we define the column vector $\bfa_\bfb$ whose entries are indexed by the faces of $\calc$.
	For each face $F$ of $\calc$, we define $\bfa_\bfb(F) = 1$ if restricting $\bfb$ to the indices given by $F$ contains no $1$s;
	otherwise we define $\bfa_\bfb(F) = 0$.
	The matrix $A$ whose columns are the set $\{\bfa_\bfb\}_{\bfb\in\{0,1\}^n}$ has the same toric ideal as $\cala_\calc$.
	In particular, $A$ is normal if and only if $\cala_\calc$ is normal.
\end{defn}
See \cite{sullivant2010} for more details about the above construction, including proof of correctness and generalization to arbitrary $\bfd$.
For the rest of this section, we let $\cala_\calc$ denote the matrix described in Definition \ref{def:fullrank}.
We now illustrate this construction with an example.
\begin{ex}
	Let $\calc = [1][23]$ on ground set $\{1,2,3\}$.
	Then the faces of $\calc$ are $\{\emptyset,\{1\},\{2\},$  $\{3\},\{2,3\}\}$.
	Then $\cala_\calc$ can be represented as follows
	\[
		\begin{pmatrix}
			1&1&1&1&1&1&1&1\\
			1&1&1&1&0&0&0&0\\
			1&1&0&0&1&1&0&0\\
			1&0&1&0&1&0&1&0\\
			1&0&0&0&1&0&0&0
		\end{pmatrix}.
	\]
\end{ex}

\begin{prop}\label{almost}
	Let $\calc$ be a simplicial complex and let $S$ be a minimal non-face.
	Let $\cald$ be the simplicial complex obtained from $\calc$ by including $S$ as a facet.
	We define $B$ to be the matrix for the $\mathcal{H}$-representation of $\rr_{\ge 0}\cala_\cald$;
	that is, $B$ satisfies
	\[
		\rr_{\ge 0} \cala_\cald = \{x \in \rr^d : Bx \ge 0\}.
	\]
	There is exactly one column in $B$ corresponding to the face $S$.
	If this column has entries contained in $\{0,+1,-1\}$,
	then if $\cald$ is normal, so is $\calc$.
\end{prop}
\begin{proof}
	This is a special case of Lemma { 2} in \cite{sullivant2010}.
\end{proof}

\begin{ex}
For example, for the complex $\calc $ with facets $\{1\}, \{2,3\}$,
the matrix $B$ is
$$
B = 
\begin{pmatrix}
0 & 1 & 0 & 0 & 0 \\
1 & -1 & 0 & 0 & 0 \\
0 & 0 & 0 & 0 & 1  \\
0 & 0 & 0 & 1 & -1  \\
0 & 0 & 1 & 0 & -1 \\
1 &  0 & -1 & -1 & 1
\end{pmatrix}.
$$
Applying Proposition \ref{almost} we see that the complex $\calc$ with
facets $\{1\}, \{2\}, $ and $\{3\}$ is also normal.
\end{ex}


\section{Minimally Non-Normal Binary Hierarchical Models}\label{sec:minimal}

Using normality- and unimodularity-preserving operations to obtain the classification results in \cite{sullivant2010}
	and \cite{bernstein-sullivant2015} required a complete list of the minimal non-normal and non-unimodular complexes.
	The list in \cite{sullivant2010} was particularly nice; it included just one complex.
	The list in \cite{bernstein-sullivant2015} included one infinite family and six other complexes.
	With our given list of operations (vertex deletion, edge contraction,
	reducibility, taking a cone)
	the list of minimally non-normal complexes appears to be even larger.
	This section describes our current knowledge of theoretical
	results towards describing this list.  
	In particular, we report
	on two infinite families of minimally non-normal complexes and a great
	many other complexes which may or may not fit into infinite families. 

	If $\mathcal{A}_{\calc,\bfd}$ is normal and $\bfd' \le \bfd$, then $\mathcal{A}_{\calc,\bfd'}$ is also normal.
	Therefore, a good first step towards a general classification of normal $\mathcal{A}_{\calc\,\bfd}$
	would be a classification of the case where $\bfd = (2,2,\dots, 2)$.
	We now restrict our attention to matrices $\mathcal{A}_{\calc,\bfd}$ where $\bfd = (2,2,\dots,2)$
	and so we abbreviate $\mathcal{A}_{\calc,(2,2,\dots,2)}$ by $\cala_\calc$.
	Since the matrix $\mathcal{A}_\calc$ depends entirely on the simplicial complex parameter,
	we abuse notation and say ``$\calc$ is normal'' when we mean ``$\cala_\calc$ is normal.''
	We give a list of all the complexes that we know to be minimally non-normal,
	a notion we make precise with the following definition.
	
	\begin{defn}
		Let $\mathcal{C}$ be a simplicial complex.
		We say that a non-reducible complex $\mathcal{C}$ without cone and ghost vertices
		is \emph{minimally non-normal} if $\mathcal{C}$ is not normal,
		but any edge-contraction, vertex-deletion, or link of $\mathcal{C}$ is normal.
	\end{defn}
	
	By our general setup at the beginning of Section \ref{sec:preliminaries},
	the indexing set of the columns of a matrix $\cala_\calc$ is $[2]\times[2]\times\dots\times[2]$
	and the rows are indexed according to this set.
	However, since we are working in the binary case,
	it feels more natural to index with $0-1$ strings than with $1-2$ strings,
	so we instead index with the set $\{0,1\}^n$.
	This also allows us to use the notions of Hamming weight and Hamming distance in their natural form.
	\\
	\indent
	In \cite{sullivant2010}, the second author shows that the class of normal graphs is minor-closed
	with $K_4$, the complete graph on four vertices, as the unique forbidden minor.
	Theorem \ref{example} below gives a family of non-normal simplicial complexes that generalize $K_4$ as a type of minimal forbidden minor.
	
	\begin{thm}\label{example}
		Let $\mathcal{C}$ be the simplicial complex on the $n+1$ vertices $\{0,1,\dots,n\}$
		whose facets are the set $\{0,n\}$ and all $n-2$-dimensional sets of vertices that do not include both $0$ and $n$.
		Then $\mathcal{A}_{\calc}$ is minimally non-normal.
	\end{thm}
	
	Before proving Theorem \ref{example}, we must set some notation and prove { two propositions}.
	\begin{defn}\label{hamming}
		If $a,b \in \zz_2^n$ are binary $n$-tuples, then we define the \emph{Hamming distance} between $a$ and $b$ as follows
		\[
			d_H(a,b) := \#\{i : a_i \neq b_i\}.
		\]
		That is, $d_H(a,b)$ gives the number of indices in which $a$ and $b$ are different.
		We define the \emph{Hamming weight} of a binary $n$-tuple $a$ to be
		\[
			w_H(a) := \#\{i: a_i = 1\} = d_H(a,0).
		\]
		That is $w_H(a)$ is the number of nonzero entries in $a$.
{
	\\
		If $\bfi = (i_0,i_1,\dots,i_n)$ is a binary $n+1$-tuple, define
		\[
			\bfi^{\setminus k} := (i_0,\dots \hat{i_k}, \dots,i_n),
		\]
		i.e. the binary $n$-tuple that results when we omit the $k$th entry.
		If $B$ is a collection of binary $n$-tuples, then we define
		\[
			B^{\setminus k} := \{\bfi^{\setminus k} : \bfi \in B\}.
		\]}
	\end{defn}
	
	\begin{prop}\label{parity}
		Given a collection $B$ of $2^{n-1}$ binary $n$-tuples, the following two conditions are equivalent:
		\begin{enumerate}
			\item\label{item:dist} For all $a,b \in B$, $d_H(a,b) \ge 2$
			\item\label{item:weight} The Hamming weight of all elements in $B$ have the same parity.
		\end{enumerate}
	\end{prop}
	{
	\begin{proof}
		It is clear that (\ref{item:weight}) implies (\ref{item:dist}).
		So assume that $B$ satisfies (\ref{item:dist}).
		Then $B^{\setminus 1}$ consists of all $2^{n-1}$ binary $n-1$ tuples,
		as no two binary $n$ tuples of the same Hamming weight parity can disagree only in the first position.
		Without loss of generality, assume $(0,\dots,0) \in B$.
		Now we show that all tuples in $B$ have even hamming weight.
		For $\bfi \in B$, let $k_\bfi$ denote the number of ones in the final $n-1$ entries of $\bfi$.
		If $\bfi \in B$ has $k_\bfi > 0$, then for some $j > 1$, $\bfi_j = 1$.
		Let $\bfi'$ be the binary $n$ tuple that agrees with $\bfi$ everywhere except $\bfi_j' = 0$.
		Then ${\bfi'}^{\setminus 1} = {\bfi''}^{\setminus 1}$ for some $\bfi'' \in B$.
		Since $d_H(\bfi,\bfi') = 1$, $\bfi_1'' = 1+\bfi_1 (mod 2)$.
		So either the hamming weights of $\bfi,\bfi''$ are the same, or they differ by $2$.
		By induction on $k_\bfi$,
		the Hamming weight of $\bfi''$ is even.
		Therefore, so is the Hamming weight of $\bfi$.
	\end{proof}
	}
	
	When $A = \mathcal{A}_{\calc}$ for some $\calc$,
	checking that some $\bfb \in \rr_{\ge 0} A$ is in $\zz A$ is easy.
	The following proposition shows that we only need to check that $\bfb$ has integer entries.
	
	\begin{prop}\label{lattice}
		For any simplicial complex $\mathcal{C}$,
		the matrix $\mathcal{A}_\calc$ has a full-rank submatrix with determinant 1.
	\end{prop}
	\begin{proof}
		{ We assume that $\mathcal{A}_\calc$ is presented as in Definition \ref{def:fullrank}.}
		Proposition 2.2. in \cite{sullivant2010} gives that the rank of $\mathcal{A}_\calc$ is the number of faces in $\mathcal{C}$.
		We now give a procedure for generating the desired submatrix.
		Let $F_1,\dots,F_k$ be an ordering of the facets of $\calc$.
		Let $f_1^i,\dots,f_r^i$ denote the faces of $F_i$ that are not contained in any earlier facet.
		For each $f_j^i$, choose the column of $\mathcal{A}_\calc$ whose index has zeros at exactly the vertices of {$f_j^i$}.
		This submatrix is upper triangular with $1$s along the diagonal.
	\end{proof}

	We are now ready to prove Theorem \ref{example}.
	
	\begin{proof}[Proof of Theorem \ref{example}]
		We first prove the minimality claim.
		Deleting $0$ or $n$ yields $\partial \Delta_{n-1}$;
		deleting any other vertex yields $\Lambda(\Delta_{n-3} \sqcup \Delta_0)$.
		The link about $0$ or $n$ is {$G(\partial \Delta_{n-2})$};
		the link about any other vertex has reducible decomposition $(\partial\Delta_{n-2}, \Delta_{n-3},\partial\Delta_{n-2})$.
		Contracting an edge $[0n], [0k]$ or $[kn]$ for $0 < k < n$ yields $\partial\Delta_{n-1}$;
		contracting any other edge yields $\Lambda(\Delta_{n-3} \sqcup \Delta_0)$.
		All of these resulting complexes are normal.
		\\
		\indent
		As stated in the beginning of this section,
		we can index the set of columns of $\mathcal{A}_\calc$ by the set of binary $n+1$-tuples, $\{0,1\}^{n+1}$.
		We show that $\calc$ is not normal by showing that the following vector is a hole
		\[
			\bfb = \sum_{\bfi\in \{0,1\}^{n+1}} \frac{1}{4} \bfa_\bfi.
		\]
		Recall that we can index each entry of $\bfa_\bfi,\bfb$ by a facet $F$ of $\calc$
		and a binary tuple whose entries are indexed by the vertices of $F$.
		The entry of $\bfa_\bfi$ indexed by $(F,\bfi_F)$ is $1$ if the binary representation of $i$
		agrees with $\bfi_F$ at the indices indicated by $F$, and $0$ otherwise.
		If $F$ is an $n-2$-dimensional facet (i.e. has $n-1$ vertices),
		then for any fixed $\bfi_F$, there are exactly { four indices} $\bfi$ such that $\bfa_\bfi$ has a $1$ at the entry
		corresponding to $(F,\bfi_F)$.
		So each $(F,\bfi_F)$ entry in $\bfb$ is $1$, and in particular, an integer.
		For the edge $[0n]$, then for any fixed $\bfi_{[0n]}$,
		there are exactly $2^{n-1}$ { indices} $\bfi$ such that $\bfa_\bfi$ has a $1$ at the entry
		corresponding to $([0n],\bfi_{[0n]})$.
		So an entry in $\bfb$ corresponding to $([0n],\bfi_{[0n]})$ is $\frac{2^{n-1}}{4}$,
		which is integral when $n \ge 3$.
		So when $n \ge 3$, $\bfb \in \rr_{\ge0}A \cap \zz^d$
		and so $\bfb \in \rr_{\ge0}A \cap \zz A$ by Proposition \ref{lattice}.
		\\
		\indent
		We now show that $\bfb \notin \nn A$.
		For the sake of contradiction, assume there exists some $\bfn \in \nn^{2^{n+1}}$ such that $A\bfn = \bfb$.
		If $\bfn_\bfi > 0$ then for each $n-2$-dimensional facet $F$,
		$\bfa_\bfi$ adds $\bfn_\bfi$ to exactly one entry indexed by a pair $(F,\bfi_F)$.
		Since each such entry is $1$ in $\bfb$, $\bfn$ must be a $0-1$ vector.
		Furthermore, exactly $2^{n-1}$ entries of $\bfn$ must be $1$ since
		there are exactly $2^{n-1}$ distinct $(F,\bfi_F)$ for each fixed $n-2$-dimensional facet $F$.
		So the vector $\bfn$ describes a collection of $2^{n-1}$ binary $n+1$-tuples, which we denote $B$, i.e.
		\[
			\cala_\calc \bfn = \sum_{\bfi \in B} \bfa_\bfi.
		\]

		Since $\calc$ includes each $F$ that is an $n-1$-element subset of $\{0,\dots,n-1\}$
		and each $(F,\bfi_F)$ entry of $\bfb$ is $1$,
		we must have ${d_H(\bfi^{\setminus n}, \bfj^{\setminus n})} \ge 2$ for every $\bfi,\bfj \in B$.
		So Proposition \ref{parity} implies that all the elements of { $B^{\setminus n}$}
		have the same parity.
		So for $\bfi \in B$, we must have $i_0 = i_1 + \dots + i_{n-1}\textnormal{ } (mod 2)$.
		Since $\calc$ includes each $F$ that is an $n-1$-element subset of $\{1,\dots,n\}$
		and each $(F,\bfi_F)$ entry of $\bfb$ is $1$,
		the same argument shows that $i_n = i_1 + \dots + i_{n-1} \textnormal{ }(mod 2)$.
		So $i_0 = i_n$ for all $\bfi \in B$.
		But then the entries of $\bfb$ corresponding to $([0n],(0,1))$ and $([0n],(1,0))$
		must be $0$.
		But we know that these entries are $2^{n-3}$.
	\end{proof}

	We can use the forbidden-minor classification of unimodular simplicial complexes
	given in \cite{bernstein-sullivant2015} to give some more examples of minimally non-normal complexes.
	Namely, if $\mathcal{C}$ is normal, but minimally non-unimodular, then its Lawrence lifting $\Lambda(\calc)$ is minimally non-normal.
	The minimally non-unimodular complexes include the infinite family $\{\partial\Delta_n \sqcup \{v\}\}_{n\ge 1}$, and six other complexes.
	We know which of these are normal, and which are not.
	Computations in Normaliz \cite{normaliz} show that the complexes $P_4, J_1, J_1^*$ (defined below) are normal.
	We can also see that $\partial\Delta_n \sqcup \{v\}$ is normal by Lemma \ref{reducible} 
	since $\partial\Delta_n$ and $\{v\}$ are normal (unimodular even - see \cite{bernstein-sullivant2015}),
	and $\partial\Delta_n \sqcup \{v\}$ has reducible decomposition $(\partial\Delta_n,\emptyset,\{v\})$.
	
	\begin{prop}\label{bigfacetnonnormal}
		Let $\mathcal{C}$ be among the following complexes.
		\begin{enumerate}
			\item $P_4$, the path on $4$ vertices
			\item $J_1$, the complex $[12][15][234][345]$
			\item $J_1^*$ the complex $[134][235][245]$ (this is the Alexander dual of $J_1$)
			\item $\partial\Delta_n \sqcup \{v\}$ for $n \ge 1$.
		\end{enumerate}
		Then $\Lambda(\calc)$ is minimally non-normal.
	\end{prop}
	\begin{proof}
		Let $u$ denote the vertex added when creating $\Lambda(\calc)$ from $\calc$.
		If we delete $u$ we are left with a simplex, which is normal.
		If we take the link of the $u$, then we are left with $\calc$ which is normal.
		If we take the link or deletion of any other vertex $w \neq u$, we are left with a Lawrence lifting of a link or deletion of $\calc$, call it $\calc'$.
		Since $\calc$ is \emph{minimally} non-unimodular, $\calc'$ is unimodular and so $\Lambda(\calc')$ is normal.
		If we contract an edge that connects to $u$, then we are left with a simplex, with is normal.
		If we contract an edge that does not connect to $u$, then we are left with a Lawrence lifting of { an} edge contraction of $\calc$, call it $\calc'$.
		We can check that $\calc'$ is unimodular and so $\Lambda(\calc')$ is normal.
	\end{proof}
	
	The remaining three minimally non-unimodular complexes are not normal.
	The complex $J_2 = [12][235][34][145]$ is not normal but not minimally so -
	contracting the edge $[14]$ gives $\Lambda(\partial\Delta_1 \sqcup \{v\})$.
	The other two minimally non-unimodular complexes are minimally non-normal,
	as we show in the following proposition.
	
	\begin{prop}
		The following complexes are minimally non-normal:
		\begin{enumerate}
			\item $[123][134][145][125][623][634][645][625]$, the boundary of the octahedron
			\item $[1234][3456][1256]$, the Alexander dual of the boundary of the octahedron.
		\end{enumerate}
	\end{prop}
	\begin{proof}
		A computation in Normaliz \cite{normaliz} shows that the first complex is not normal.
		The matrix for the second complex is the matrix for the no 3-way interaction model with $r_1=r_2=r_3 = 4$,
		which is not normal \cite[Theorem 6.4]{ohsugi-hibi2007}.
		All links and induced sub-complexes are unimodular \cite{bernstein-sullivant2015} and therefore normal.
		Contracting an edge in the first complex gives either the cone over a square or the Alexander dual of $\Delta_2\sqcup \Delta_1$,
		both of which are unimodular.
		Contracting an edge in the second complex gives a complex whose matrix is the matrix for the no 3-way interaction model
		with $r_1=r_2 = 4$ and $r_3 = 2$ which is normal \cite[Theorem 6.4]{ohsugi-hibi2007}.
	\end{proof}


\section{Compressed Hierarchical Models}\label{sec:compressed}

A stronger property for a matrix $A \in \zz^{d \times n}$
is to be compressed.  This condition guarantees normality
and can be easier to verify directly through polyhedral computations.
As in the normal case, we show that taking induced subcomplexes, edge contractions, reducible decompositions, and links
preserves the property of being compressed.

To state the definition of compressed vector configuration we
first need the definition of the toric ideal associated to the
configuration $A$.  See \cite{sturmfels} for more details on toric ideals and
their Gr\"obner bases.

\begin{defn}
Let $A \in \zz^{d \times n}$ and $\kk[x] :=  \kk[x_{1}, \ldots, x_{n}]$.
The \emph{toric ideal} associated to the matrix $A$ is the ideal
$$
\langle x^{u} - x^{v} :  u,v \in \nn^{n}, Au = Av  \rangle.
$$
\end{defn}

\begin{defn}
	If $A \in \zz^{d\times n}$ is an integer matrix such that the corresponding toric ideal $I_A \subset \kk[x_1,\dots,x_n]$ is homogeneous.
	Then we say that $A$ is \emph{compressed} if the initial ideal $in_{\prec}(I_A)$ is generated by squarefree monomials
	whenever $\prec$ is a reverse lexicographic term order.
\end{defn}

It can often be easier to use the following characterization of compressed matrices.

\begin{thm}[\cite{sullivant2006}]\label{compressedEquiv}
	Let $A \in \zz^{d\times n}$ be an integer matrix such that the corresponding toric ideal $I_A \subset \kk[x_1,\dots,x_n]$ is homogeneous
	and let $B$ be the matrix of facet defining inequalities of $\rr_{\ge 0}A$.
	In other words,  $B$ satisfies
	\[
		\rr_{\geq 0}A  =  \{Ax: x \ge 0\} = \{y: By \ge 0\}.
	\]
	Suppose that the entries in each column of $B$ are relatively prime.
	Then $A$ is compressed if and only if $AB^{T}$ is a $0/1$ matrix.
\end{thm}

As the following Proposition states, compressed-ness is a stronger property than normality, and weaker than unimodularity.
Hence, to classify the compressed models, we can restrict attention to normal models.

\begin{prop}\label{compressedNormal}
	If $A \in \zz^{d \times n}$ is compressed, then $A$ is normal.
	If $A \in \zz^{d \times n}$ is unimodular, then $A$ is compressed.
\end{prop}

\begin{prop}\label{opcompressed}
	The following operations on simplicial complexes preserve the property of being compressed of $\mathcal{A}_{\calc,\bfd}$:
	\begin{itemize}
		\item gluing two complexes along a {common} facet
		\item passing to induced subcomplexes
		\item adding or removing cone vertices
		\item contracting edges.
	\end{itemize}
\end{prop}
\begin{proof}
	See \cite{sullivant2006}.
\end{proof}

We also add to the list of compressed-ness operations computing a link
of a vertex.

\begin{prop}\label{substitute}
	Let $A\in \zz^{d\times n}$ be a matrix with columns $\bfa_1,\dots,\bfa_n$.
	{ Assume that the entries of $a_n$ do not have a common divisor}.
	Let $A'$ be the matrix consisting of the first $n-1$ columns of $P_{\bfa_n}A$.
	Let $\phi:\kk[x_1,\dots,x_n] \rightarrow \kk[x_1,\dots,x_{n-1}]$ be the substitution homomorphism mapping $x_n$ to $1$.
	Then $I_{A'} = \phi(I_A)$.
\end{prop}
\begin{proof}
	Note that $u \in \ker_\zz(A')$ iff $P_{\bfa_n}A\binom{u}{0} = 0$
	which is true iff $Au = k \bfa_n$ for some $k \in \zz$ which is true iff $u-(0,\dots,0,k)^\intercal \in \ker_\zz(A)$.
	Now let $\bfx^{u^+}-\bfx^{u-}$ be a generator of $I_{A'}$.
	Since $u-(0,\dots,0,k)^\intercal \in \ker_\zz(A)$, we have $\bfx^{u^+}-x_n^k\bfx^{u-} \in I_A$ without loss of generality.
	Then note $\bfx^{u^+}-\bfx^{u-} = \phi(\bfx^{u^+}-x_n^k\bfx^{u-})$.
\end{proof}

\begin{thm}\label{linkcompressed}
	Assume $\mathcal{A}_{\mathcal{C},{\bf d}}$ is compressed.
	Let $v$ be a vertex of $\mathcal{C}$ and let ${\bf d}'$ denote the vector obtained by deleting the entry for $v$ from ${\bf d}$.
	Then $\mathcal{A}_{\link_v{\mathcal{C}},{\bf d}'}$ is compressed.
\end{thm}

\begin{proof}
	Let $A = \mathcal{A}_{\link_v{\mathcal{C}},{\bf d}'}$ and let $B = \mathcal{A}_{\mathcal{C}\setminus v, {\bf d}'}$.
	As in the proof of Theorem \ref{link}, we know that we can write
		\[
			\mathcal{A}_{\mathcal{C},{\bf d}} =\left[\begin{matrix}
			A &  {\bf 0}  & \ldots & {\bf 0}\\
			{\bf 0}  &  A & \ldots & {\bf 0}\\
			\vdots & \vdots & \ddots & \vdots\\
			{\bf 0}  &   {\bf 0}       &\ldots & A\\
			B & B & \dots & B
			\end{matrix}\right]
		\]
	and that projecting onto the space orthogonal to the last $(d-1)n$ columns and deleting unnecessary rows,
	then deleting columns of zeros, leaves the matrix $A$.
	
	Let $\kk[\bfx]$ be the ring containing $I_A$, and $\kk[\bfx,\bfy]$ be the ring containing $I_{\mathcal{A}_{\calc,\bfd}}$.
	By Proposition \ref{substitute}, $I_A$ is obtained from $I_{\mathcal{A}_{\calc,\bfd}}$
	by plugging in $1$ for all the $\bfy$ variables.
	Since $I_A$ is the toric ideal for a hierarchical model, it is homogeneous.
	Consider any reverse lexicographic order $\prec$ on $\kk[\bfx]$.
	Extend this to a reverse lexicographic order $\prec'$ on $\kk[\bfx,\bfy]$
	by putting the $\bfy$ variables ahead of the $\bfx$ variables in the variable order (not in the term order).
	This means that if $\bfx^{u^+}-\bfx^{u^-} \in I_A$ has initial term $\bfx^{u^+}$ with respect to $\prec$, then for any $\bfv$
	such that $\bfy^{v^+}\bfx^{u^+}-\bfx^{u^-}\bfy^{v^-} \in I_{\mathcal{A}_{\calc,\bfd}}$,
	then $\bfy^{v^+}\bfx^{u^+}$ is the initial term with respect to $\prec'$.

Since	$\mathcal{A}_{\calc,\bfd}$ is compressed, there is a Gr\"obner basis
element $\bfy^{\bar{v}^+}\bfx^{\bar{u}^+}-\bfx^{\bar{u}^-}\bfy^{\bar{v}^-} \in  I_{\mathcal{A}_{\calc,\bfa}}$  
whose leading term $\bfy^{\bar{v}^+}\bfx^{\bar{u}^+}$ is a 
squarefree monomial and divides
$\bfy^{v^+}\bfx^{u^+}$.  Without loss of generality, we can assume that
$\bar{u}^{+} \neq 0$.  Indeed, any $\bfy^{\bar{v}^+}\bfx^{\bar{u}^+}-\bfx^{\bar{u}^-}\bfy^{\bar{v}^-} \in  I_{\mathcal{A}_{\calc,\bfa}}$ 
with $\bar{u}^{+} = 0$ must also satisfy $\bar{u}^{-} = 0$, and so we can
reduce $\bfy^{v^+}\bfx^{u^+}-\bfx^{u^-}\bfy^{v^-}$ with respect to such
polynomials to obtain a polynomial where $\bfy^{v^+}$ is reduced with respect
to the Gr\"obner basis of $I_{\mathcal{A}_{\calc,\bfd}}$.

Now apply the map $\phi: \kk[\bfx, \bfy] \rightarrow \kk[\bfx]$
that sets all $y$ variables to $1$.  Then the leading term of 
$\phi(\bfy^{\bar{v}^+}\bfx^{\bar{u}^+}-\bfx^{\bar{u}^-}\bfy^{\bar{v}^-})
= \bfx^{\bar{u}^+}-\bfx^{\bar{u}^-}$
is $\bfx^{\bar{u}^+}$ which divides $\bfx^{u^+}$.  This implies that
the initial terms in any Gr\"obner basis element of $I_{A}$ with respect
to the reverse lexicographic order are squarefree.  Hence, 
	$\mathcal{A}_{\link_v{\mathcal{C}},{\bf d}'}$ is compressed.
\end{proof}


\section{Computational Results}\label{sec:computations}

In this section we summarize the results of our computations
to classify the normality and compressed-ness of $\cala_{\calc}$ for complexes
with few vertices.  In particular, we highlight the
simplicial complexes whose normality and compressed-ness, or lack thereof, cannot be determined
from our existing construction methods for normal complexes or our
examples of minimally nonnormal complexes. 
In particular, we wrote a Mathematica code to generate,
up to symmetry, all simplicial complexes for which normality and compressed-ness cannot { be} ruled out
by looking at induced subcomplexes, vertex links, edge contraction, 
cone vertices, or reducibility.  From the resulting lists we excluded 
all simplicial complexes that are graphs (where normality is already classified by
\cite{sullivant2010}) or that are of Lawrence type (where normality
follows from the classification of unimodular complexes \cite{bernstein-sullivant2015}).  Then we use Normaliz \cite{normaliz} to decide
normality of the remaining complexes in our lists.  The results
are summarized below. 

\subsection{Four and Fewer Vertices}
Every simplicial complex on $1,2,$ or $3$ vertices is normal
and all of these follow from the construction methods described
above.  On $m = 4$ vertices, the normality or nonnormality
of all simplicial complexes can be decided by the methods listed
above.  Among all complexes on $4$ vertices, the only two that
are nonnormal are the following:
\begin{enumerate}
\item  $[12][13][14][23][24][34]$
\item  $[123][14][24][34]$
\end{enumerate}
both of which are clearly minimally nonnormal.
The compressed complexes on four vertices are exactly the normal ones.

\subsection{Five Vertices}
Applying the procedure described above to
simplicial complexes on $5$ vertices produces $14$ complexes
which the techniques mentioned above do not decide the normality of.
Of these $14$ complexes, $6$ were normal.
These are
\begin{enumerate}
\item $[12][13][245][345]$
\item $[12][134][235][345]$
\item $[123][124][135][245]$
\item $[12][134][135][234][235]$
\item $[123][124][134][235][245]$
\item $[123][124][134][235][245][345]$.
\end{enumerate}
The first complex is normal because it is equivalent to the four-cycle
with $\bfd = (2,2,2,4)$, so normality follows from Theorem \ref{thm:4cycle}.
The last complex is normal because it is a unimodular complex,
in particular it is the Alexander dual of $\Delta_{1} \sqcup \Delta_{2}$.
For the other complexes, normality was verified using Normaliz.

The remaining $8$ out of $14$ complexes are all nonnormal.  These are
\begin{enumerate}
			\item $[123][124][135][245][345]$
			\item $[123][124][125][134][234][345]$
			\item $[123][124][125][134][235][345]$
			\item $[12][134][135][145][234][235][245]$
			\item $[123][124][125][134][135][234][245]$
			\item $[123][124][125][134][135][245][345]$
			\item $[12][134][135][145][234][235][245][345]$
			\item $[123][124][125][134][135][234][245][345]$.
\end{enumerate}
The second to last complex on this list is verified to be non-normal
as a special case of Theorem \ref{example}. 
{ The other complexes on the list}
were shown to be non-normal using Normaliz.
On five vertices, the minimal non-normal complexes consist of these
eight simplicial complexes, plus the two complexes of Lawrence type
\begin{enumerate}
\item $[1234][125][235][345]$
\item $[1234][15][235][245][345]$.
\end{enumerate}

We can use the results from $4$-vertex complexes
and Theorem \ref{linkcompressed} and Proposition \ref{opcompressed}
to show that every normal complex on $5$ vertices is compressed, aside from the 5-cycle
\[
	[12][23][34][45][15].
\]

\subsection{Six Vertices}

We applied the above procedure to simplicial complexes on $6$ vertices.
This produced $80$ complexes for which the techniques above do not determine normality.
Using some other techniques, described below, we were able to determine
normality and compressedness of all of these complexes.
The results are described in the table below.

\newpage

\footnotesize
\begin{center}
\begin{longtable}{|l|l|l|l|l|} 
\hline \multicolumn{1}{|c|}{} & \multicolumn{1}{c|}{\textbf{Complex}} & \multicolumn{1}{c|}{\textbf{Normal?}} & \multicolumn{1}{c|}{\textbf{Notes}} & \multicolumn{1}{c|}{\textbf{Compressed?}} \\ \hline 
\endfirsthead

\multicolumn{5}{c}%
{{\bfseries \tablename\ \thetable{} -- continued from previous page}} \\
\hline \multicolumn{1}{|c|}{} &
\multicolumn{1}{c|}{\textbf{Complex}} &
\multicolumn{1}{c|}{\textbf{Normal?}} &
\multicolumn{1}{c|}{\textbf{Notes}} &
\multicolumn{1}{c|}{\textbf{Compressed?}} \\ \hline 
\endhead

\hline \multicolumn{5}{|r|}{{Continued on next page}} \\ \hline
\endfoot

\hline \hline
\endlastfoot
	$0$ & $[1456][1236][2345]$&No& O &No\\ \hline
	$1$ & $[2346][16][2345][15]$&Yes& O&Yes\\ \hline
	$2$ & $[2346][16][2345][125]$&Yes & AC&Yes\\ \hline
	$3$ & $[236][16][2345][145]$&Yes& C&Yes\\ \hline
	$4$ & $[2346][16][2345][1235]$&Yes & AC&Yes\\ \hline
	$5$ & $[3456][126][1345][123]$&Yes & AR&Yes\\ \hline
	$6$ & $[346][126][345][125]$&Yes & AR&Yes\\ \hline
	$7$ & $[1346][126][345][125]$&Yes& C&Yes\\ \hline
	$8$ & $[1346][126][1345][235]$&Yes& C&Yes\\ \hline
	$9$ & $[3456][126][1345][1234]$&Yes& C&Yes\\ \hline
	$10$ & $[1246][1236][2345][1345]$&Yes & AR&Yes\\ \hline
	$11$ & $[26][16][345][245][13]$&Yes & AR&No\\ \hline
	$12$ & $[236][16][245][235][14]$&Yes & AR&No\\ \hline
	$13$ & $[2456][236][16][1245][123]$&Yes & AR&Yes\\ \hline
	$14$ & $[236][16][245][235][124]$&Yes & AC&No\\ \hline
	$15$ & $[236][16][245][145][234]$&Yes& C&No\\ \hline
	$16$ & $[2356][2346][16][1235][1234]$&Yes & AC&Yes\\ \hline
	$17$ & $[346][126][345][135][123]$&Yes & AR&No\\ \hline
	$18$ & $[2456][136][126][2345][123]$&Yes & AR&Yes\\ \hline
	$19$ & $[3456][2356][1346][126][123]$&Yes & AC&Yes\\ \hline
	$20$ & $[136][126][345][245][124]$&Yes& C&Yes\\ \hline
	$21$ & $[136][126][2345][1345][124]$&Yes & AR&Yes\\ \hline
	$22$ & $[1356][346][126][1345][124]$&Yes & AR&Yes\\ \hline
	$23$ & $[2356][1346][126][2345][124]$&Yes& C&Yes\\ \hline
	$24$ & $[3456][1346][126][1345][125]$&Yes & AR&Yes\\ \hline
	$25$ & $[2456][1356][1246][1236][125]$&Yes & AR&Yes\\ \hline
	$26$ & $[2356][1346][126][2345][1234]$&Yes & AC&Yes\\ \hline
	$27$ & $[3456][1346][126][2345][1234]$&Yes& C&Yes\\ \hline
	$28$ & $[3456][1246][1236][1345][1234]$&Yes & AR&Yes\\ \hline
	$29$ & $[3456][1246][1236][1245][1235]$&No&C&No\\ \hline
	$30$ & $[26][16][245][145][23][13]$&Yes& C&Yes\\ \hline
	$31$ & $[456][356][246][136][126][14]$&Yes & AR&No\\ \hline
	$32$ & $[246][236][16][245][235][15]$&Yes & AR&Yes\\ \hline
	$33$ & $[236][16][235][125][234][124]$&Yes & AC&Yes\\ \hline
	$34$ & $[3456][2456][136][126][234][123]$&Yes & AR&Yes\\ \hline
	$35$ & $[246][136][126][245][235][123]$&Yes & AR&No\\ \hline
	$36$ & $[3456][2456][136][126][2345][123]$&Yes & AR&Yes\\ \hline
	$37$ & $[2346][1346][126][2345][1345][123]$&Yes & AC&Yes\\ \hline
	$38$ & $[2356][1356][2346][1346][126][123]$&Yes & AR&Yes\\ \hline
	$39$ & $[136][126][235][125][134][124]$&Yes & AR&Yes\\ \hline
	$40$ & $[136][126][245][235][134][124]$&Yes & AR&No\\ \hline
	$41$ & $[136][126][345][245][134][124]$&Yes& C&Yes\\ \hline
	$42$ & $[3456][2456][136][126][134][124]$&Yes& C&Yes\\ \hline
	$43$ & $[136][126][145][135][234][124]$&Yes & AR&No\\ \hline
	$44$ & $[136][126][235][135][234][124]$&Yes & AR&Yes\\ \hline
	$45$ & $[136][126][345][135][234][124]$&Yes & AR&No\\ \hline
	$46$ & $[2346][136][126][2345][1345][124]$&Yes& C&Yes\\ \hline
	$47$ & $[3456][2356][1346][126][2345][124]$&Yes& C&Yes\\ \hline
	$48$ & $[2346][1346][126][2345][1345][125]$&Yes& C&Yes\\ \hline
	$49$ & $[246][136][126][345][245][135]$&No&C&No\\ \hline
	$50$ & $[1356][2346][1346][126][1235][1234]$&Yes & AR&Yes\\ \hline
	$51$ & $[3456][2356][1346][126][2345][1234]$&Yes & AC&Yes\\ \hline
	$52$ & $[2456][1356][1246][1236][1345][1234]$&No&C&No\\ \hline
	$53$ & $[2356][1346][1246][1236][2345][1234]$&Yes & AR&Yes\\ \hline
	$54$ & $[1346][1246][1236][2345][1245][1235]$&Yes & AR&Yes\\ \hline
	$55$ & $[2456][1356][1246][1236][2345][1345]$& { No}  & C & { No}
	\\ \hline
	$56$ & $[346][246][136][126][245][125][14]$&Yes & AR&No\\ \hline
	$57$ & $[456][356][256][146][136][126][15]$&Yes & AR&Yes\\ \hline
	$58$ & $[356][246][136][126][235][234][123]$&Yes & AR&Yes\\ \hline
	$59$ & $[356][246][136][126][345][234][123]$&Yes & AR&No\\ \hline
	$60$ & $[346][246][136][126][345][245][123]$&Yes & AR&No\\ \hline
	$61$ & $[136][126][135][125][234][134][124]$&Yes & AR&Yes\\ \hline
	$62$ & $[246][136][126][345][245][135][125]$&No&C&No\\ \hline
	$63$ & $[3456][126][2345][1345][1245][1235][1234]$&No&C&No\\ \hline
	$64$ & $[2356][1356][2346][1346][126][1235][1234]$&Yes & AR&Yes\\ \hline
	$65$ & $[3456][2356][1346][126][2345][1345][1245]$&{ No}  & C & { No}\\ \hline
	$66$ & $[1456][1356][1246][1236][2345][1345][1234]$&No&C&No\\ \hline
	$67$ & $[2456][1356][1246][1236][2345][1345][1234]$& { No}  & C & { No}\\ \hline
	$68$ & $[2456][2356][1346][1246][1236][2345][1234]$&Yes & AR&Yes\\ \hline
	$69$ & $[456][356][246][136][126][345][234][123]$&Yes & AR&No\\ \hline
	$70$ & $[136][126][235][135][125][234][134][124]$&Yes & AR&Yes\\ \hline
	$71$ & $[346][246][136][126][345][245][135][125]$&No&C&No\\ \hline
	$72$ & $[1456][1356][2346][1346][126][2345][1345][125]$& { No}  & C & { No}\\ \hline
	$73$ & $[3456][2456][2356][1346][1246][1236][2345][1234]$&Yes & U&Yes\\ \hline
	$74$ & $[2456][2356][1346][1246][1236][1345][1245][1235]$&No&C&No\\ \hline
	$75$ & $[3456][2456][2356][1346][1246][1236][2345][15][1234]$& { No}  & C & { No}\\ \hline
	$76$ & $[236][136][126][235][135][125][234][134][124]$&Yes & AR&Yes\\ \hline
	$77$ & $[3456][1456][1356][2346][1346][126][2345][1345][125]$& 
	{ No}  & C & { No}\\ \hline
	$78$ & $[3456][2456][2356][1346][1246][1236][1345][1245][1235]$&Yes &U&Yes\\ \hline
	$79$ & $[3456][2456][2356][2346][16][2345][1345][1245][1235][1234]$&No&O&No
\end{longtable}
\end{center}

\normalsize

\begin{itemize}
\item[U:] Complex is unimodular.
\item[AC:] Complex is proved normal using  Proposition \ref{almost} and
the fact that one face can be added to produce a cone.
\item[AR:] Complex is proved normal using  Proposition \ref{almost} and
the fact that one face can be added to produce a reducible complex.
\item[C:]  Complex was proved normal, or not normal, by a computation in
Normaliz.
\item[O:]  Normality/Nonnormality determined by other method or theoretical result.
\end{itemize}

The complexes indicated by an O had their normality or nonnormality
decided by a theoretical result mentioned in the paper.
Complex $1$ has the same matrix as $\mathcal{A}_{\calc,\bfd}$ where $\calc$ is the $4$-cycle and $\bfd = (2,2,2,8)$,
which is normal by Theorem \ref{thm:4cycle}.
Complexes $73$ and $78$ are unimodular and therefore normal.
Complex $0$ has the same matrix as $\mathcal{A}_{\calc,\bfd}$ where $\calc$ is the $3$-cycle and $\bfd = (4,4,4)$
which is not normal by Theorem \ref{triangle}.
Complex $79$ is an instance of the minimally non-normal family from Theorem \ref{example}.
{ For complexes 55, 65, 67, 72, 75, and 77 our computations with Normaliz
were initially unsuccessful.   The Normaliz team  was able to confirm nonnormality of these examples using the current development version.}

Proposition \ref{almost} was used to prove normality in a number of
cases where we added a face and either produced a reducible complex
or a cone complex that was known to be normal. These are indicated in the 
table with an AR or AC respectively.
With our expanded list of complexes known to be normal, we can apply this same technique over and over again
until we no longer add any new complexes to the set that we know to be normal.
This technique only gives us three more complexes - adding the facet $[236]$ to complex $27$ yields
$37$, adding $[23]$ to $41$ yields $56$ and adding $[134]$ to $20$ yields $41$.
In all cases, the added face satisfies the conditions of Lemma \ref{almost}, and so complexes $27,41,$ and $20$ are normal.


{

\section{Future Directions}

Results from the previous
section suggest that a complete classification of complexes $\calc$ such that
$\cala_\calc$ in normal is probably extremely complicated.  
At any rate, our current tools
for proving normality or nonnormality are inadequate for completing a classification.
This suggest the following problems:

\begin{pr}
Develop new construction techniques for producing complexes $\calc$ such that
$\cala_\calc$ is normal.
\end{pr}

The technique based on Proposition \ref{almost} remains mysterious and application
of it in more settings (and generalizations of it) would require an
understanding of the facet defining inequalities of $\rr_{\geq 0}  \cala_\calc$
for families of simplicial complexes.

\begin{pr}
Develop new methods for constructing holes in the semigroups $\nn \cala_\calc$.
\end{pr}

In particular, it would be worthwhile  to analyze the holes produced in the nonnormal
examples found above and try to generalize them to find infinite families
of minimal nonnormal complexes.

While a complete classification of the complexes $\calc$ 
such that $\cala_\calc$ is normal is
out of reach at present, it might be reasonable to try to classify low dimensional
complexes that yield normal vector configurations.  For example, the 
set of $1$-dimensional $\calc$ such that $\cala_\calc$ is normal
is extremely simple (the set of $K_4$-minor free graphs \cite{sullivant2010}), so perhaps
$2$-dimensional complexes $\calc$ such that $\cala_\calc$ are normal
also have a simpler description.  As a starting point, we propose
the following:

\begin{pr}\label{pr:2dim}
Classify the two-dimensional manifolds $\calc$ such that $\cala_\calc$
is normal.
\end{pr}

In principle, the solution to Problem \ref{pr:2dim} should yield
a relatively simple list of complexes.  Indeed, a $2$-dimensional manifold
with many triangles seems likely to have a minor isomorphic to $K_4$ or the
complex $[123][14][24][34]$, so there seem to be few possibilities.

}

\section*{Acknowledgments}
{Thanks to Winfried Bruns and Christof S\"oger for computations with
the development version of Normaliz that
allowed us to complete the classification of normal complexes on six vertices.}
Daniel Irving Bernstein was partially supported by the US National Science Foundation (DMS 0954865).
Seth Sullivant was partially supported by the David and Lucille Packard 
Foundation and the US National Science Foundation (DMS 0954865). 

{\footnotesize
\bibliography{normal}}
\bibliographystyle{plain}

\end{document}